\documentclass[11 pt]{amsart}
\usepackage{amsmath,amssymb,amsthm}

\usepackage[pdftex]{graphicx}
\usepackage[margin=2.8cm]{geometry}

\usepackage{xcolor}
\usepackage{hyperref} 
\usepackage{pbox}
\usepackage{verbatim} 

\newtheorem{theorem}{Theorem}[section]

\newtheorem{corollary}[theorem]{Corollary}

\newtheorem{definition}[theorem]{Definition}

\newtheorem{lemma}[theorem]{Lemma}

\newtheorem{proposition}[theorem]{Proposition}

\DeclareMathOperator{\area}{area}
\DeclareMathOperator{\conv}{conv}
\DeclareMathOperator{\vol}{vol}
\DeclareMathOperator{\orbit}{\mathcal{O}}
\DeclareMathOperator{\Stab}{Stab}

\DeclareMathOperator{\relint}{relint}

\newcommand{\be}{\begin{equation}}
\newcommand{\ee}{\end{equation}}
\newcommand{\bea}{\begin{eqnarray}}
\newcommand{\eea}{\end{eqnarray}}

\newcommand{\bc}{\begin{center}}
\newcommand{\ec}{\end{center}}
\newcommand{\ben}{\begin{enumerate}}
\newcommand{\een}{\end{enumerate}}
\newcommand{\bfi}{\begin{figure}}
\newcommand{\efi}{\end{figure}}

\title[Generalized permutahedra]{Generalized permutahedra: Minkowski linear functionals and Ehrhart positivity}
\date{}
\author{Katharina Jochemko}
\address{Department of Mathematics, %
KTH Royal Institute of Technology, %
Sweden}
\email{jochemko@kth.se}
\author{Mohan Ravichandran}
\address{Department of Mathematics, %
Bogazici University, %
Istanbul, Turkey}
\email{mohan.ravichandran@boun.edu.tr}

\begin{document}
\maketitle
\renewcommand{\arraystretch}{3.0}
\setlength{\unitlength}{1cm}

\begin{abstract}
We characterize all signed Minkowski sums that define generalized permutahedra, extending results of Ardila-Benedetti-Doker (2010). We use this characterization to give a complete classification of all positive, translation-invariant, symmetric Minkowski linear functionals on generalized permutahedra. We show that they form a simplicial cone and  explicitly describe their generators. We apply our results to prove that the linear coefficients of Ehrhart polynomials of generalized permutahedra, which include matroid polytopes, are non-negative, verifying conjectures of De Loera-Haws-K\"{o}ppe (2009) and Castillo-Liu (2018) in this case. We also apply this technique to give an example of a solid angle polynomial of a generalized permutahedron that has negative linear term and obtain inequalities for beta invariants of contractions of matroids. 
\end{abstract}

\section{Introduction}
Generalized permutahedra form a combinatorially rich class of polytopes that naturally appear in many areas of mathematics such as combinatorics, geometry, representation theory, optimization,  game theory and statistics (see, e.g.,~\cite{Ardila,CastilloFu,Danilov,fink2018schubert,Goemans,Lange,GraphicalModels,RankTests,Faces,Postnikov}). They contain a variety of interesting and significant classes of polytopes, in particular, matroid polytopes. Generalized permutahedra are sufficiently special to admit a thorough combinatorial description of their geometry as witnessed for instance by the discovery of Aguiar-Ardila of a Hopf monoid structure on generalized permutahedra~\cite{Hopfmonoid}, but also general enough to be widely applicable and to serve as useful test cases for questions in polyhedral combinatorics. In recent years, different groups of authors have explored generalizations of this class, leading to generalized nested permutahedra~\cite{nested} and generalized Coxeter permutahedra~\cite{ACEP20}. 

The name \textit{generalized permutahedra} was introduced by Postnikov in his pioneering work on the combinatorial aspects of this interesting class of polytopes~\cite{Postnikov}. It should however be noted that generalized permutahedra are equivalent to \textit{polymatroids}, a class of polyhedra that were introduced by Edmonds~\cite{Edmonds} in 1970 as polyhedral generalization of matroids. Since then polymatroids have been intensively studied in optimization, game theory and statistics due to their correspondence to submodular and supermodular functions (see~\cite{Polymatroids,M-convex,Shapley}.) For example, in game theory, well-studied objects are cooperative games, to each of which a polytope called the \emph{core} of the game is associated, see~\cite{Shapley2, StudenyKroupa}. Generalized permutahedra turn out to be exactly equal to cores of convex cooperative games~\cite{Kuipers}. In the theory of discrete convex analysis~\cite{M-convex} $M$-convex sets play a central role and there is a one to one correspondence between lattice points of generalized permutahedra and $M$-convex sets. For a thorough discussion of the equivalence of these concepts as well as connections to further areas such a conditional independence structures, we refer the reader to~\cite{StudenyKroupa}.

Recall that the (standard) permutahedron $\Pi _d\subset \mathbb{R}^d$ is the $(d-1)$-dimensional polytope 
\[
\Pi _d \ = \ \conv \{ (\sigma (1),\sigma (2),\ldots, \sigma (d))\colon \sigma \in S_d\} \subset \mathbb{R}^d
\]
where $S_d$ denotes the group of permutations on $[d]=\{1,2,\ldots,d\}$. There are many  equivalent  ways of defining generalized permutahedra, the most concise one being via Minkowski summands of the permutahedron. The \textbf{Minkowski sum} of two polytopes $P,Q\subset \mathbb{R}^d$ is the polytope defined as the vector sum 
\[P+Q=\{p+q \colon p\in P, q\in Q\}.\]

A polytope $R \subset \mathbb{R}^d$ is called a \textbf{Minkowski summand} of another polytope $Q \subset \mathbb{R}^d$ if there is a polytope $P \subset \mathbb{R}^d$ such that $P + R = Q$. We also call $R$ the \textbf{Minkowski difference} of $Q$ and $P$ and use the notation $R=Q-P$.  Further, the polytope $R$ is called a \textbf{weak Minkowski summand} of $Q$ if it is a Minkowski summand of a dilate $\lambda Q$ for some $\lambda>0$.

\begin{definition}\label{def:genperm}
A polytope $P \subset \mathbb{R}^d$ is called a \textbf{generalized permutahedron} if it is a weak Minkowski summand of the permutahedron $\Pi _d$.
\end{definition} 
In the following we denote the class of all generalized permutahedra in $\mathbb{R}^d$ by $\mathcal{P}_d$. In particular, every generalized permutahedron $P\in \mathcal{P}_d$ is a polytope of dimension at most $d-1$ and is contained in a hyperplane $\{\mathbf{x}\in \mathbb{R}^d\colon \sum _{i=1}^dx_i =\ell\}$ for some $\ell \in \mathbb{R}$.

In~\cite{Postnikov}, Postnikov studied the subclass of generalized permutahedra consisting of Minkowski sums of dilated standard simplices. Let $\Delta _\emptyset = \{0\}$ and for $\emptyset \neq I\subseteq [d]$ let
\[
\Delta _I \ = \ \conv \{e_i \colon i\in I\} \, 
\]
 be the \textbf{standard simplices} where  $e_1,\ldots, e_d$ are the standard basis vectors in $\mathbb{R}^d$. We will also use the notation $\Delta _i$ to denote the $(i-1)$-dimensional simplex $\Delta _{[i]}$ for all $1\leq i\leq d$. Extending \cite[Proposition 6.3]{Postnikov}, Ardila, Benedetti and Doker~\cite[Proposition 2.4]{Ardila} proved that every generalized permutahedron is a Minkowski difference of sums of dilated standard simplices and can be uniquely represented as a signed Minkowski sum $\sum _{I\subseteq[d]}y_I\Delta _I$. This representation was also considered in earlier works by Danilov and Koshevoy~\cite{Danilov} where it was used to describe cores of cooperative games. Here, a \textbf{signed Minkowski sum} is a formal linear combination with coefficients $y_I\in \mathbb{R}$ that describes a Minkowski difference: 
 
 \[\sum _{I\subseteq[d]}y_I\Delta _I = \sum _{I\subseteq[d], y_I \geq 0}y_I\Delta _I - \sum _{I\subseteq[d], y_I<0}(-y_I)\Delta _I.\]
 
 Not every set of coefficients $\{y_I\}_{I\subseteq [d]}$ defines a generalized permutahedron, though, as we will see, the set of all possible coefficients forms a polyhedral cone. In Theorem~\ref{thm:chargenperm} we give an explicit inequality description of this cone, thereby characterizing all coefficients $\{y_I\}_{I\subseteq [d]}$ that define generalized permutahedra. We moreover prove that this cone is equal to the cone of \textbf{supermodular functions}, up to a change of coordinates. Interestingly, Theorem~\ref{thm:chargenperm} has appeared in a very different context and language within game theory: it can be seen as a reincarnation of a result by Kuipers, Vermeulen and Voorneveld~\cite[Theorem 9]{Kuipers} who characterized all convex games given as a linear combination in the so-called unanimity basis. We offer a geometric proof of this result.

 We then use the characterization obtained in Theorem~\ref{thm:chargenperm} to investigate \textbf{Minkowski linear functionals} on generalized permutahedra. In Theorem~\ref{prop:positivecomb} and Proposition~\ref{prop:positivegeom} we explicitly describe the rays of the cone of positive Minkowski linear functionals and provide an explicit geometric construction of the ray functionals. We then consider Minkowski linear functionals that are \textbf{symmetric}, that is, invariant under permutations of the coordinates. Minkowski linear functionals are valuations and structural results on valuations under the action of a group have been a focal point of research in classical convex geometry ever since Hadwiger's seminal classification of continuous, rigid-motion invariant valuations on convex bodies~\cite{Hadwiger}. In Theorem~\ref{thm:symcomb} we provide a complete classification of all positive, translation-invariant, symmetric Minkowski linear functionals: they form a simplicial cone and we explicitly determine the rays of this cone. We then apply our results to Ehrhart polynomials of  generalized permutahedra that are also lattice polytopes.
 
 The \textbf{Ehrhart polynomial} of a lattice polytope counts the number of lattice points in integer dilates of the polytope~\cite{Ehrhart}. It is appealing to view Ehrhart polynomials as discrete analogues of the classical Minkowski volume polynomials of convex bodies~\cite{Bernstein,JochemkoSanyal,McMullen}, but unlike volume polynomials, the coefficients of Ehrhart polynomials need not be nonnegative. Understanding when we do have positivity is a fundamental question in Ehrhart theory (see, e.g., \cite{BeckRobins,negativeCoefficients}) and the study of \textbf{Ehrhart positive}~\cite{CastilloFu} polytopes, namely those that have only nonnegative coefficients is of current particular interest.

Known examples of Ehrhart positive polytopes include zonotopes~\cite{zonotopes} and integral cyclic polytopes~\cite{FuCyclic}. However, there are elementary examples of non-Ehrhart positive polytopes, the most classical being the Reeve tetrahedron~\cite{Reeve}. In recent work, it has been shown that order polytopes~\cite{negative0} and smooth polytopes~\cite{negative1} need not be Ehrhart positive. For a comprehensive survey on Ehrhart positivity see~\cite{liu2019positivity}.

In~\cite{CastilloFu} Castillo and Liu conjectured Ehrhart positivity for generalized permutahedra expanding on a conjecture of De Loera, Haws and Koeppe on matroid polytopes~\cite{DeLoera}. The conjecture was known to hold for all sums of standard simplices by an explicit combinatorial formula given in~\cite{Postnikov}. Ferroni~\cite{ferronihypersimplices} showed that hypersimplices, that is, matroid polytopes of uniform matroids, are Ehrhart positive. Using a valuation theoretic approach Castillo and Liu~\cite{CastilloFu} proved that generalized permutahedra are Ehrhart positive in up to six dimensions and moreover showed that the third and the fourth highest coefficient are nonnegative for generalized permutahedra of any dimension. However, despite this evidence, both of the aforementioned conjectures have very recently, while this article was under review, simultaneously been disproved by Ferroni~\cite{ferroni2021matroids} who was able to construct examples of matroid polytopes with negative quadratic coefficients for all ranks greater or equal to three.

On the other hand, in~\cite{CastilloThesis,CastilloFu} strong computational evidence was given that the linear coefficient is always nonnegative by explicit calculations for $d\leq 500$. Using the classification of positive, symmetric, translation-invariant Minkowski linear functionals obtained in Theorem~\ref{thm:symcomb} we are able to prove in Theorem~\ref{thm:Ehrhartpositive} that the linear coefficient is indeed always nonnegative. This has independently also been shown by Castillo and Liu~\cite{castillo2019todd} using different techniques from those developed in the present article. As an application, we then obtain an inequality among beta invariants of contractions of any given matroid in Corollary~\ref{cor:matroids} using a result of Ardila, Benedetti and Doker~\cite{Ardila}. Further, we prove that the aforementioned formula for the number of lattice points in sums of standard simplices provided in~\cite{Postnikov} extends to arbitrary generalized permutahedra (Corollary~\ref{cor:extpostnikov}). We conclude by applying our results to solid-angle polynomials and show the existence of a three dimensional generalized permutahedron whose solid-angle polynomial has negative linear term.

\section{Signed Minkowski sums}

In the following we assume familiarity with the basics of polyhedral geometry and lattice polytopes. For further reading we recommend~\cite{BeckRobins,Grunbaum,Ziegler}.

Let $P_1,\ldots,P_m$ be polytopes. A \textbf{signed Minkowski sum} is a formal sum $\sum _i y_i P_i$ with real coefficients $y_1,\ldots, y_m$. We say that $\sum _i y_i P_i$ defines a polytope if $P=\sum _{i\colon y_i<0} (-y_i)P_i$ is a Minkowski summand of $Q=\sum _{i\colon y_i\geq 0} y_iP_i$, in which case $\sum _i y_i P_i$ represents the Minkowski difference $Q-P$. In~\cite{Ardila}, Ardila, Benedetti and Doker showed that every generalized permutahedron has a unique expression as a signed Minkowski sum of standard simplices. This decomposition was also considered in earlier works by Danilov and Koshevoy~\cite{Danilov} where it was used to describe cores of cooperative games.
\begin{proposition}[{\cite[Proposition 2.4]{Ardila}}]\label{prop:Ardila}
For every generalized permutahedron $P\in \mathcal{P}_d$ there are uniquely determined real numbers $y_I$ for all $\emptyset \neq I\subseteq [d]$ and $y_\emptyset=0$ such that 
\[
P \ = \ \sum_{\emptyset \neq I\subseteq [d]} y_I \Delta _I \, .
\]
\end{proposition}
Equivalently, $\sum _{I\colon y_I<0} (-y_I) \Delta _I$ is a Minkowski summand of $\sum _{I\colon y_I\geq 0} y_I \Delta _I$ and 
\begin{equation}\label{eq:equivalent}
    P \ + \ \sum _{I\colon y_I<0} (-y_I) \Delta _I \ = \ \sum _{I\colon y_I\geq 0} y_I \Delta _I \, .
\end{equation}
Not every choice of coefficients $\{y_I\}_{I\subseteq [d]}$ yields a generalized permutahedron. The goal of this section is to complete the picture and to give a complete characterization of all coefficients $\{y_I\}_{I\subseteq [d]}$ for which $\sum _{I\subseteq [d]}y_I\Delta _I$ defines a generalized permutahedron.

By a result of Shephard, Minkowski summands of polytopes can be characterized in terms of their edge directions and edge lengths (see~\cite[p. 318]{Grunbaum}). For any polytope $P\subset \mathbb{R}^d$ and any direction $u\in \mathbb{R}^d\setminus \{0\}$ let
\[
P^u = \{x \in P \mid u^{T}x = \operatorname{max}_{y \in P}\,u^{T}y\} \, 
\]
be the \textbf{face of $P$ in direction of $u$}.
\begin{theorem}[{\cite[p. 318]{Grunbaum}}]\label{thm:Minkowskisummands}
Let $P,Q\subset\mathbb{R}^d$ be polytopes. Then $P$ is a Minkowski summand of $Q$ if and only if the following two conditions hold for all $u\in \mathbb{R}^d\setminus \{0\}$.
\begin{itemize}
    \item[(i)] If $Q^u$ is a vertex then so is $P^u$.
    \item[(ii)] If $Q^u=[p,q]$ is an edge with endpoints $p$ and $q$ then up to translation, $P^u=\lambda \,[p,q]$ for some $0\leq \lambda \leq 1$.
\end{itemize}
\end{theorem}
From Theorem~\ref{thm:Minkowskisummands} it follows that the possible edge directions of a Minkowski summand $P$ of $Q$ are given by the edge directions of $Q$. Since the permutahedron $\Pi_d$ equals, up to translation, the Minkowski sum over all line segments $[e_i,e_j]$, $i\neq j$ (See, e.g., \cite[Exercises 4.63 and 4.64]{Stanley}), all edge directions of $\Pi _d$ are of the form $e_i-e_j$ for $i\neq j$. This property characterizes generalized permutahedra  as shown by Proposition 2.6 in \cite{ACEP20}, specialized to the permutohedron.

\begin{theorem}[{\cite[Proposition 2.6]{ACEP20}}]\label{thm:edgedirections}
A polytope is a generalized permutahedron if and only if all edge directions are of the form $e_i-e_j$ for $i\neq j$.
\end{theorem}
The following theorem characterizes all signed Minkowski sums that define generalized permutahedra.
It was brought to the authors' attention by the anonymous referee that this theorem has appeared before in a different language in the game theory literature in an article by Kuipers-Vermuelen-Voorneveld~\cite{Kuipers}. There it yields a characterization of the class of convex games in terms of the unanimity basis introduced by Shapley in~\cite{Shapley2}. We offer two proofs: the second one, via supermodular functions, is similar in nature to the one given in~\cite{Kuipers}. Nevertheless, for reasons of completeness and to highlight the connection to supermodular functions, we have chosen to include it. Our first proof, in contrast, is, up to our knowledge, new and rather different in spirit, and offers a geometric perspective on this result.

In the following let ${[d]\choose 2}$ denote the set of all subsets of $[d]$ with $2$ elements.

\begin{theorem}\label{thm:chargenperm}
Let $\{y_I\}_{I\subseteq [d]}$ be a vector of real numbers. Then the following are equivalent.
\begin{itemize}
    \item[(i)] The signed Minkowski sum $\sum_{I \subseteq [d]} y_I \Delta_I$ defines a generalized permutahedron in $\mathcal{P}_d$.
    \item[(ii)] For all $2$-element subset $E\in {[d]\choose 2}$ and all $T\subseteq [d]$ such that $E\subseteq T$
    \begin{equation}\label{eq:supermodtrans}
    \sum _{E\subseteq I\subseteq T}y_I  \geq 0 \, .
    \end{equation}
In particular, the collection of all coefficients $\{y_I\}_{I\subseteq [d]}$ such that $\sum_{I \subseteq [d]} y_I \Delta_I$ defines a generalized permutahedron is a polyhedral cone. The inequalities \eqref{eq:supermodtrans} are facet-defining.
\end{itemize}
\end{theorem}
\begin{proof}
Let $\alpha _I=-\min \{y_I,0\}$ and $\beta _I = \max \{y_I,0\}$ and let $P=\sum _{I} \alpha_I \Delta _I$ and $Q=\sum_{I} \beta_I \Delta _I$. Then, by \eqref{eq:equivalent}, we need to show that $P$ is a Minkowski summand of $Q$ if and only if
\begin{equation}\label{eq:inequality}
\sum _{E\subseteq I\subseteq T}\alpha _I \ \leq \ \sum _{E\subseteq I\subseteq T}\beta _I
\end{equation}
for all $2$-element subsets $E$ of $[d]$ and all $T\subseteq [d]$ such that $E\subseteq T$. 

We first prove the necessity of the inequality. Let $E=\{i, j\}$ and let $T\supseteq E$. Let $u\in \mathbb{R}^d\setminus \{0\}$ be a vector such that
\begin{itemize}
    \item $u_i = u_j$ and $u_k \neq u_l$ for $k \neq l$ with $\{k, l\} \neq \{i, j\}$, and
    \item  further,
\[\operatorname{min}_{k \notin  T} u_k > u_i=u_j > \operatorname{max}_{k \in T \setminus E} u_k.\] 
\end{itemize}
A calculation shows that for such a vector $u$, the face $\Delta_I^u$ is either a point or an edge, 
\[\Delta_I^{u} = \begin{cases}
[e_i, e_j], & \text{ if } E  \subseteq I \subseteq T,\\
e_k, & \text{ if } \text{otherwise,} \, \text{ where } k = \underset{k \in I}{\operatorname{arg max}} \,\,u^{T}e_k.
\end{cases}.\]
Therefore up to translation, 
\[
P^u \ = \ \sum _{I} \alpha_I \Delta _I^u \ = \ \sum_{E \subseteq I \subseteq T} \alpha _I [e_i,e_j],
\]
and
\[
Q^u \ = \ \sum _{I} \beta_I \Delta _I^u \ = \ \sum_{E \subseteq I \subseteq T} \beta _I [e_i,e_j].
\]
Thus the desired inequality follows from Theorem~\ref{thm:chargenperm}. 

For the converse direction, assume that $u\in \mathbb{R}^d\setminus \{0\}$ is a vector such that $Q^u$ is either a vertex or an edge. Let us first assume that $Q^u$ is a vertex. We claim that $P^u$ must also be a vertex. To see this, assume otherwise there is an $I$ with $\alpha _I>0$ and $\dim \Delta _I^u>0$. Then $[e_i,e_j]\subseteq \Delta _I^u$ for some $i,j\in I$, $i\neq j$. This further implies that $[e_i,e_j]\subseteq \Delta _J^u$ for all $\{i,j\}\subseteq J\subseteq I$. By \eqref{eq:inequality},
\[
0 \ < \ \alpha _I \ \leq \ \sum _{\{i,j\}\subseteq J\subseteq I}\alpha _I \ \leq \ \sum _{\{i,j\}\subseteq J\subseteq I}\beta _I \, .
\]
Thus there must be a $\{i,j\}\subseteq J\subseteq I$ with $\beta _J>0$ and therefore $\dim Q^u\geq \dim \Delta _J ^u > 0$, a contradiction.

If $Q^u$ is an edge, by Theorem~\ref{thm:edgedirections}, we may assume that $Q^u=\lambda [e_i,e_j]$ for some $\lambda > 0$, up to translation. Then necessarily, $u_i=u_j$. Let $M$ be the subset of all $2$-element subsets $\{k,l\}$ for which $u_k=u_l$. For all $F=\{k,l\}\in M$ let $T_{F}=\{i\in [d]\colon u_i\leq u_k=u_l\}$. We observe that $[e_k,e_l]\subseteq \Delta _I ^u$ if and only if $F\subseteq I \subseteq T_F$. Therefore, for all $F\neq E$ in $M$ and all $I$ with $F\subseteq I\subseteq T_F$ we must have $\beta _I=0$ since $Q^u=\lambda [e_i,e_j]$. Thus we also obtain
\[
\sum _{F\subseteq I\subseteq T_F}\beta _I \ = \ 0 \, ,
\]
and by \eqref{eq:inequality} this equality remains true if we replace all $\beta _I$ by $\alpha _I$. This, in turn, implies that $P^u$ equals $\mu [e_i,e_j]$ with $\mu=\sum _{E\subseteq I\subseteq T_{E}}\alpha _I $ which by \eqref{eq:inequality} is smaller than $\lambda=\sum _{E\subseteq I\subseteq T_{E}}\beta _I $. Thus $P$ is a Minkowski summand of $Q$ by Theorem~\ref{thm:Minkowskisummands}.

Theorem~\ref{prop:positivecomb} below together with its proof via cone duality imply that the inequalities \eqref{eq:supermodtrans} are facet-defining.
\end{proof}

The previous proof of Theorem~\ref{thm:chargenperm} made use of the characterization of the edge directions of generalized permutahedra given in Theorem~\ref{thm:edgedirections}. We now give a second proof that will display that the inequalities \eqref{eq:supermodtrans} given in Theorem~\ref{thm:chargenperm} are exactly the defining inequalities of the cone of supermodular functions after a change of variables.

In what follows, we use the notation $2^{[d]}$ to denote the set of all subsets of $[d]$. A function $2^{[d]} \rightarrow \mathbb{R}, I\mapsto z_I$  is called \textbf{supermodular} if 
\begin{equation}\label{eq:supermod}
z_I + z_J \ \leq \ z_{I\cup J} + z_{I\cap J}
\end{equation}
for all subsets $I, J \subseteq [d]$. In particular, the set of all supermodular functions forms a polyhedral cone. This cone has been in the focus of research in game theory, statistics and optimization. In optimization, typically the equivalent perspective of submodular functions is taken: a function $f$ is \textbf{submodular} if and only if $-f$ is supermodular. The facets of the pointed cone of supermodular functions, normalized such that $z_\emptyset =0$, are well-understood and are given by all inequalities of the form
\begin{equation}\label{eq:supermodfacets}
    z_{K\cup \{i\}} +z_{K\cup \{j\}} \ \leq \  z_{K\cup \{i,j\}}+z_{K} \,
\end{equation}
for all $K\subseteq [d]$ and all $i,j\in [d]\setminus K$, $i\neq j$, (see, e.g., \cite[Theorem 44.1]{CombOptB})). In contrast, the rays of the cone of supermodular functions are far less understood. In~\cite{Shapley} Shapley gave an explicit description of the rays in the case $d=4$. Also Edmonds~\cite{Edmonds} raised the question of determining the extreme submodular functions. In~\cite{studeny2016basic} operations preserving the rays are studied and in~\cite{StudenyKroupa} necessary and sufficient conditions for extremality of a supermodular function are given. For further references on extreme supermodular/submodular functions as well as their significance in the pertaining areas we refer to~\cite{StudenyKroupa}.

There is a one-to-one correspondence of supermodular functions and generalized permutahedra via their facet description: for every vector $\{z_I\}_{I\subseteq [d]} \in \mathbb{R}^{2^{[d]}}$ with $z_\emptyset =0$ let 
\[
P(\{z_I\}) \ = \ \left\{ \mathbf{x}\in \mathbb{R}^d \colon \sum _{i=1}^d x_i =z_{[d]} \, , \sum _{i\in I} x_i \geq z_I \text{ for all } \emptyset \subseteq I\subset [d] \right\} \, ,
\]
where we assume that all $z_I$ are chosen maximally, that is, all defining inequalities of the polytope $P(\{z_I\})$ are tight. Every generalized permutahedra in $\mathcal{P}_d$ is a polytope of the form $P(\{z_I\})$, but not every such polytope is a generalized permutahedra. The following theorem characterizes all vectors $\{z_I\}$ for which $P(\{z_I\})$ is a generalized  permutahedron. This characterization appeared in~\cite[Proposition 3.2]{Faces}. The equivalence to Definition~\ref{def:genperm} follows from Lemma 9 and Corollary 11 in~\cite{StudenyKroupa}.  
\begin{theorem}\label{thm:inequalities}
Let $\{z_I\}_{I\subseteq [d]}$ be a vector in $\mathbb{R}^{2^{[d]}}$ with $z_\emptyset =0$. Then the polytope $P(\{z_I\})$ is a generalized permutahedron if and only if the function $2^{[d]}\rightarrow \mathbb{R}, I \mapsto z_I$
is supermodular.
\end{theorem}

In~\cite{Ardila}, Ardila, Benedetti and Doker explicitly described the representation of $P(\{z_I\})$ as signed Minkowski sum.
\begin{proposition}[{\cite[Proposition 2.4]{Ardila}}]\label{prop:Ardila}
For every generalized permutahedron $P(\{z_I\})\in \mathcal{P}_d$ there are uniquely determined real numbers $y_I$ for all $\emptyset \neq I\subseteq [d]$ and $y_\emptyset =0$ such that
\[
P(\{z_I\}) \ = \ \sum_{I\subseteq [d]} y_I \Delta _I \, ,
\]
namely $y_I=\sum _{J\subseteq I} (-1)^{|I|-|J|} z_J$.
\end{proposition}
\begin{proof}[Second proof of Theorem~\ref{thm:chargenperm}]
Let $U$ be the linear transformation defined by
\begin{align*}
 U\colon \mathbb{R}^{2^{[d]}} & \longrightarrow \mathbb{R}^{2^{[d]}} \\
  z_I & \longmapsto
  y_I = \sum _{J\subseteq I}(-1)^{|I|-|J|}z_J \, .
\end{align*}
Then, by M\"obius inversion, $U$ is a bijection with $z_I=U^{-1}(y_I)=\sum _{J\subseteq I}y_J$ for all $I$. By Theorem~\ref{thm:inequalities}, $P(\{z_I\})$ is a generalized permutahedron if and only if $\{z_I\}$ satisfies the supermodularity condition \eqref{eq:supermod}. On the other hand, by Theorem~\ref{prop:Ardila}, $P(\{z_I\})=\sum y_I\Delta _I$ where $y_I= \sum _{J\subseteq I}(-1)^{|I|-|J|}z_J=U(z_I)$. In particular, a signed Minkowski sum $\sum y_I \Delta _I$ defines a generalized permutahedron if and only if $\{y_I\}=U(\{z_I\})$ where $\{z_I\}$ satisfies the supermodularity condition \eqref{eq:supermod}. In other words, the set of all vectors $\{y_I\}$ such that $\sum y_I \Delta _I$ defines a generalized permutahedron is a polyhedral cone, namely the image of the cone of supermodular functions under the linear bijection $U$. 
By \eqref{eq:supermodfacets} $\{z_I\}$ defines a supermodular function if and only if for all $K\subseteq [d]$ and all $i,j\in [d]\setminus K$, $i\neq j$,
\[
z_{K\cup \{i\}} +z_{K\cup \{j\}} \ \leq \  z_{K\cup \{i,j\}}+z_{K} \, .
\]
These inequalities are facet-defining and equivalent to
\begin{eqnarray}
\sum _{J\subseteq K\cup \{i\}}y_J +\sum _{J\subseteq K\cup \{j\}}y_J & \leq &  \sum _{J\subseteq K\cup \{i,j\}}y_J+\sum _{J\subseteq K}y_J \quad\quad\Leftrightarrow\\
 0 & \leq & \sum _{J\subseteq K}y_{J\cup \{i,j\}}\label{eq:cond} \, .
\end{eqnarray}
We conclude by observing that the inequality \eqref{eq:cond} is equivalent to condition \eqref{eq:supermodtrans} when interchanging $K$ with $T\setminus \{i,j\}$.
\end{proof}

\section{Minkowski linear functionals}
We call a function $\varphi \colon \mathcal{P}_d\rightarrow \mathbb{R}$ \textbf{Minkowski linear} if $\varphi (\emptyset)=0$ and
\[
\varphi(\lambda P+\mu Q) \ = \ \lambda \varphi (P) + \mu \varphi (Q) 
\]
for all $P,Q \in \mathcal{P}_d$ and all $\lambda,\mu \geq 0$. The function $\varphi$ is \textbf{positive} if $\varphi (P)\geq 0$ for all $P\in \mathcal{P}$ and \textbf{translation-invariant} if $\varphi (P+t)=\varphi (P)$ for all $P\in \mathcal{P}_d$ and all $t\in \mathbb{R}^d$. If $\varphi \colon \mathcal{P}\rightarrow \mathbb{R}$ is a Minkowski linear functional then by linearity we obtain
\[
\varphi \left(\sum _I y_I \Delta _I\right) \ = \ \sum _I y_I \varphi (\Delta _I) \, 
\]
and $\varphi(\Delta_\emptyset )=0$. By Theorem~\ref{prop:Ardila}, every generalized permutahedron has a unique representation as a signed Minkowski sum $\sum _I y_I\Delta _I$ given $y_\emptyset =0$. Consequently, we may identify every Minkowski linear map $\varphi \colon \mathcal{P}_d\rightarrow \mathbb{R}$ with the vector $\{\varphi (\Delta _I)\}_{\emptyset \neq I\subseteq [d]} \in \mathbb{R}^{2^{[d]}\setminus \emptyset}$.

 For any $2$-element subset $E\in {[d]\choose 2}$ and any $T\subseteq [d]$ such that $E\subseteq T$ let $v_E^T$ be the Minkowski linear functional defined by
\[
v_E^T (\Delta _I) \ = \ \begin{cases} 1 & \text{ if }E\subseteq I \subseteq T \, ,\\
0 & \text{ otherwise. }

\end{cases}
\]
Since $\Delta _{\{i\}}=e_i$  and $v_E^T (\Delta _{\{i\}})=0$ for all $1\leq i\leq d$ these functionals are translation-invariant.

The following theorem characterizes all positive, translation-invariant Minkowski linear functionals on $\mathcal{P}_d$.
\begin{theorem}\label{prop:positivecomb}
Let $\varphi \colon \mathcal{P}_d\rightarrow \mathbb{R}$ be a Minkowski linear functional. Then $\varphi$ is positive and translation-invariant if and only if there are nonnegative real numbers $c_E^T$ such that
\[
\varphi = \sum _{E\in {[d]\choose 2}}\sum _{T\supseteq E} c_E^T v_E^T \, .
\]
In particular, the family of positive, translation-invariant Minkowski linear functionals is a polyhedral cone with rays $v_E^T$.
\end{theorem}
\begin{proof}
Let $C\subseteq \mathbb{R}^{2^{[d]}\setminus \emptyset}$ be the set of all vectors $\{y_I\}$ such that $\sum  y _I \Delta _I$ defines a generalized permutahedron. Then, by Theorem~\ref{thm:chargenperm}, $C$ is a polyhedral cone with inequality description
\[
C \ = \ \bigcap _{E\in {[d]\choose 2}}\bigcap _{T\supseteq E}\{ \{y_I \} \colon \sum _{E\subseteq I \subseteq T}y_I \geq 0 \} \, .
\]
Thus, by cone duality, a Minkowski functional $\varphi$ is positive if and only if $\varphi= \sum _{E\in {[d]\choose 2}}\sum _{T\supseteq E} c_E^T v_E^T$ for some nonnegative numbers $c_E^T$. Since $v_E^T (\Delta _I)=0$ for all $1$-element subsets $I\subseteq [d]$ the functional $\varphi$ is also translation-invariant in this case. To see that the functionals $v_E^T$ are rays of the cone of positive, translation-invariant Minkowski functionals we observe that none of them can be expressed as a positive linear combination of the others. For that assume that $v_E^T = \sum \lambda_{E'}^{T'}  v_{E'}^{T'}$ for some nonnegative $\lambda_{E'}^{T'}$. Then $\lambda_{E'}^{T'}=0$ for all $E \neq E'$ and all $T' \not \subseteq T$. From evaluating $v_E^T$ at $\Delta _T$ it follows that $\lambda_{E}^{T}=1$. Then evaluating at $\Delta _E$ yields $\lambda_{E'}^{T'}=0$ for all $(E',T')\neq (E,T)$. This finishes the proof.
\end{proof}
Next, we provide a geometric description of the ray generators $v_E^T$. Let $E=\{i,j\}\in {[d] \choose 2}$ and $T\subseteq [d]$ such that $E\subseteq T$. We say that a vector $u\neq 0$ is \textbf{compatible} with $(E,T)$ if $u_i=u_j$, all other coordinates of $u$ are different and distinct from each other, and
\[
\min _{k\not \in T}u_k>u_i=u_j>\max _{k\in T}u_k \, .
\]
\begin{proposition}\label{prop:positivegeom}
Let $E\in {[d] \choose 2}$ and $T\subseteq [d]$ such that $E\subseteq T$. Let $u \neq 0$ be compatible with $(E,T)$. Then for all $P\in \mathcal{P}_d$, $P^{u}$ is one dimensional and
\begin{equation}\label{eq:geometric}
    v_E^T (P) \ = \ \vol _1 (P^u) \, ,
\end{equation}
where $\vol _1$ denotes the normalized volume where $\vol _1 ([e_i,e_j])=1$.
\end{proposition}
\begin{proof}
Let $E=\{i,j\}$. Since $u$ is compatible we have that, up to translation, $\Pi _d ^u =\sum [e_i,e_j]^u =[e_i,e_j]$. Since every generalized permutahedron is a weak Minkowski summand of $\Pi _d$, by Theorem~\ref{thm:edgedirections}, $P^u = \lambda [e_i,e_j]$, and $\vol _1 (P^u)$ is therefore well-defined. Since $(\lambda P+\mu Q)^u=\lambda P^u+\mu Q^u$ for all polytopes $P,Q$ and all $\lambda,\mu \geq 0$, equation \eqref{eq:geometric} defines a Minkowski linear functional on $\mathcal{P}_d$. We observe that since $u$ is compatible with $(E,T)$ we have $\Delta_I^u=[e_i,e_j]$ if and only if $E\subseteq I\subseteq T$. In this case $\vol _1 (\Delta _I^u)=1$. Otherwise, $\Delta_I^u$ is a vertex and $\vol _1 (\Delta _I^u)=0$. Since every Minkowski linear function is uniquely defined by its values on $\Delta _I$ for all $I\subseteq [d]$ this finishes the proof.
\end{proof}
\subsection{Symmetric Minkowski linear functionals}
We conclude this section by classifying all positive Minkowski linear functionals that are invariant under coordinate permutations. We call such functionals \textbf{symmetric}. The natural action of the symmetric group $S_d$ on $\mathbb{R}^d$ which acts by permuting the coordinates induces an action on the class of generalized permutahedra which, in turn, induces an action on Minkowki linear functionals on generalized permutahedra by $(\sigma \cdot \varphi) (P) \ = \ \varphi (\sigma (P))$ for all $P\in \mathcal{P}_d$. Then every symmetric translation-invariant Minkowski linear functional $\varphi$ can be identified with the $(d-1)$-dimensional vector $\{\varphi (\Delta _{i+1})\}_{1\leq i\leq d-1}\in \mathbb{R}^{d-1}$. For all $1\leq k\leq d-1$ let $f_k\colon \mathcal{P}_d \rightarrow \mathbb{R}$ be the symmetric, translation-invariant Minkowski linear functional defined by
\begin{equation}
    (f_k)(\Delta _{i+1}) \ = \ {i+1 \choose 2}{d-i-1 \choose k-i} \, 
\end{equation}
for all $1\leq i\leq d-1$.

\begin{theorem}\label{thm:symcomb}
Let $\varphi \colon \mathcal{P}_d\rightarrow \mathbb{R}$ be a Minkowski linear functional. Then $\varphi$ is positive, translation- and symmetric if and only if there are real numbers $c_1,\ldots, c_{d-1} \geq 0$ such that
\[
\varphi \ = \ \sum _{k=1}^{d-1} c_k f_k \, .
\]
In particular, the family of all positive, symmetric and  translation-invariant Minkowski linear functionals forms a simplicial cone of dimension $d-1$.
\end{theorem}

\begin{proof}
By Theorem~\ref{prop:positivecomb}, $\varphi$ is a positive, Minkowski linear and translation invariant linear functional if and only if $\varphi = \sum _{E\in {[d]\choose 2}}\sum _{T\supseteq E} c_E^T v_E^T$ for nonnegative numbers $v_E^T$. If $\varphi$ is moreover invariant under permutation of the coordinates we obtain 
\begin{eqnarray}
    d! \cdot \varphi \ &=& \ \sum _{\sigma \in S_d}\sigma \cdot \varphi\\
    &=& \ \sum _{E\in {[d]\choose 2}}\sum _{T\supseteq E} c_E^T \sum _{\sigma \in S_d}\sigma \cdot v_E^T\\
    &=& \ \sum _{E\in {[d]\choose 2}}\sum _{T\supseteq E} c_E^T \cdot |\Stab (v_E^T)|\sum _{\psi \in \orbit (v_E^T)}\psi \, ,\label{eqn:cone}
    \end{eqnarray}
where $\Stab (v_E^T)=\{\sigma \in S_d\colon \sigma v_E^T=v_E^T\}$ denotes the stabilizer and $\orbit (v_E^T)=\{\sigma \cdot v_E^T\colon \sigma \in S_d\}$ denotes the orbit of $v_E^T$. We observe that if $|T|=k$ then $\orbit (v_E^T)=\{v_E^T\colon |T|=k\}$. Clearly, $\sum _{\psi \in \orbit (v_E^T)}\psi$ is symmetric. Therefore, since
\begin{eqnarray*}
\sum _{\psi \in \orbit (v_E^T)}\psi (\Delta _{i+1}) \ &=& \ \sum _{E\in {[d] \choose 2}}\sum _{T\supseteq E\atop |T|=k} v_E^T (\Delta _{i+1})\\
\ &=& \ \sum _{E\in{[d] \choose 2}\atop E\subseteq [i+1]}\sum _{T\supseteq [i+1]\atop |T|=k}1 \\
\ &=& \ {i+1\choose 2}{d-i-1\choose k-i-1} \, 
\end{eqnarray*}
we see that $f_{k-1}=\sum _{\psi \in \orbit (v_E^T)}\psi$ whenever $|T|=k$. Thus, by \eqref{eqn:cone}, every symmetric translation-invariant valuation is a nonnegative linear combination of the functionals $f_1,\ldots, f_{d-1}$ which are easily seen to be linearly independent and positive by Theorem~\ref{prop:positivecomb}. This finishes the proof.
\end{proof}

\section{Applications}
\subsection{Ehrhart positivity}
A \textbf{lattice polytope} is a polytope in $\mathbb{R}^d$ with vertices in the integer lattice $\mathbb{Z}^d$. A famous result by Ehrhart states that the number of lattice points in integer dilates of a lattice polytope is given by a polynomial~\cite{Ehrhart}.
\begin{theorem}[{\cite{Ehrhart}}]\label{thm:Ehrhart}
Let $P\subset \mathbb{R}^d$ be a lattice polytope. Then there is a polynomial $E_P$ of degree $\dim P$ such that
\[
E_P (n) \ = \ |nP\cap \mathbb{Z}^d| 
\]
for all integers $n\geq 1$.
\end{theorem}
The polynomial $E_P (n)=E_0 (P)+E_1(P)n+\cdots + E_{\dim P}(P)n^{\dim P}$ is called the \textbf{Ehrhart polynomial} of $P$. In this section we show that the linear coefficient $E_1(P)$ of the Ehrhart polynomial of every generalized permutahedra $P$ with vertices in the integer lattice is nonnegative. This has independently been proved by Castillo and Liu~\cite{castillo2019todd}. In \cite{MinkowskiValuations}, the authors make the useful observation that the linear coefficient is additive under taking Minkowski sums of lattice polytopes.
\begin{lemma}[{\cite[Corollary 23]{MinkowskiValuations}}]\label{lem:linearityEhr}
Let $P$ and $Q$ be lattice polytopes and $k,\ell\geq 0$ be integers. Then
\[
E_1(kP+\ell Q)=kE_1(P)+\ell E_1(Q) \, .
\]
\end{lemma}
Let $\mathcal{E}\colon \mathcal{P}_d\rightarrow \mathbb{R}$ be the symmetric Minkowski linear functional defined by
\[
\mathcal{E} (\Delta _{i+1}) \ = \ 1+\frac{1}{2}+\cdots + \frac{1}{i} = : h_i
\]
for all $1\leq i\leq d-1$. Then $\mathcal{E}$ agrees with $E_1$ on all generalized permutahedra that are lattice polytopes.
\begin{proposition}\label{prop:agreesEhr}
Let $P\in \mathcal{P}_d$ be a generalized permutahedron with vertices in the integer lattice. Then $\mathcal{E}(P)=E_1 (P)$.
\end{proposition}
\begin{proof}
We recall that for all $1\leq i\leq d-1$
\[
E _{\Delta _{i+1}}(n) \ = \ \left\{\mathbf{x}\in \mathbb{R}^{i+1}\times \{0\}^{d-i-1}\colon \sum _{k=1}^d x_k =n\right\} \ = {n+i\choose i} \, .
\]
In particular, $E_1 (\Delta_{i+1})= 1+\frac{1}{2}+\cdots + \frac{1}{i}=\mathcal{E}(\Delta_{i+1})$. It follows from \cite[Proposition 2.3]{Ardila} that every generalized permutahedron that is a lattice polytope is a signed Minkowski sum of standard simplices $\Delta _I$ with integer coefficients. Furthermore, $E_P(n)$ and therefore $E_1(n)$ is invariant under permutations of the coordinates. Thus, the claim follows from Lemma~\ref{lem:linearityEhr}.
\end{proof}
Thus, to prove that $E_1(P)$ is always nonnegative for any generalized permutahedron $P$, by Theorem~\ref{thm:symcomb}, we are left to prove that $\mathcal{E}=\sum _{k=1}^{d-1} c_kf_k$ for nonnegative real numbers $c_1,\ldots, c_{d-1}$. Let $A=(a_{ik})=(f_1,\ldots,f_{d-1})\in \mathbb{R}^{d-1}\times \mathbb{R}^{d-1}$ be the matrix with column vectors $f_1,\ldots,f_{d-1}$. Then
\[
a_{ik} \ = \ {i+1 \choose 2}{d-i-1 \choose k-i}
\]
and
\[
c \ = \ A^{-1}h \,  ,
\]
where $h=(h_1,\ldots,h_{d-1})^T$.
\begin{lemma}
\[
A^{-1} \ = \ \frac{(-1)^{k+j}}{{j+1\choose 2}}{d-k-1\choose j-k} \ =: \ (b_{kj})=B
\]
\end{lemma}
\begin{proof}
We calculate
\begin{eqnarray*}
(AB)_{ij} \ &=& \ \sum _{k=1}^{d-1} {i+1 \choose 2}{d-i-1 \choose k-i}\frac{(-1)^{k+j}}{{j+1\choose 2}}{d-k-1\choose j-k}\\
&=&\frac{{i+1 \choose 2}}{{j+1\choose 2}}\sum_{k=1}^{d-1}(-1)^{k+j}{d-1-i \choose j-i}{j-i\choose k-i}\\
&=&(-1)^{j-i}{d-1-i \choose j-i}\frac{{i+1 \choose 2}}{{j+1\choose 2}}\sum _{k=1}^{d-1}(-1)^{k-i}{j-i\choose k-i}\\
&=&(-1)^{j-i}{d-1-i \choose j-i}\frac{{i+1 \choose 2}}{{j+1\choose 2}}(1-1)^{j-i}\\
&=& \begin{cases} 1 & \text{ if } j=i \, .\\
0 & \text{ otherwise. }
\end{cases}
\end{eqnarray*}
That is, $AB=I_{d-1}$ and thus $A$ is invertible with inverse equal to $B$.
\end{proof}
\begin{theorem}\label{thm:Ehrhartpositive}
Let $P\in \mathcal{P}_d$ be a generalized permutahedron. Then $\mathcal{E}(P)\geq 0$.
\end{theorem}
\begin{proof}
We consider the polynomial 
\[
p_k = \sum_{j = k}^{d-1} b_{kj}t^{j}
\]
and observe that
\[\int_{0}^1\dfrac{p_k(1)- p_k(t)}{1-t}dt = \int_{0}^1\sum_{j=k}^{d-1} b_{kj} \left(1 + t + \ldots + t^{j-1}\right)dt = (Bh)_k =c_k\]
which we need to show is nonnegative. It therefore suffices to show that
\[p_k'(t) \geq 0 \]
for all $t \in [0,1]$.
Let
\[
q_k(t) = \dfrac{t^2p_k'(t)}{2} = \sum_{j=k}^{d-1} (-1)^{k+j} \binom{d-k-1}{j-k} \dfrac{t^{j+1}}{j+1} \, .
\]
Then 
\[
q_k'(t) = \sum_{j=k}^{d-1} (-1)^{k+j} \binom{d-k-1}{j-k}t^j =  \sum_{\ell=0}^{d-1-k} (-1)^{\ell} \binom{d-k-1}{\ell}t^{\ell+k} \, .
\]
We conclude by observing that
\[q_k(t)  = \int_0^t q_k'(t)dt\]
and
\[q_k'(t) = t^{k}(1-t)^{d-k-1}\]
which is nonnegative for all $t\in [0,1]$.
\end{proof}

An important subclass of generalized permutahedra consists of polytopes that can be written as Minkowski sums of standard simplices. Postnikov~\cite[Theorem 11.3]{Postnikov} gave a combinatorial formula for the number of lattice points in generalized permutahedra contained in this subclass that shows Ehrhart positivity in this case (see Equation~\eqref{eq:posPostnikov} below). In the remainder of this section we will see that this formula extends to signed Minkowski sums and thus to arbitrary generalized permutahedra.

A \textbf{valuation} on lattice polytopes is a function $\varphi$ such that 
\[
\varphi (P\cup Q) = \varphi (P)+\varphi (Q)-\varphi (P\cap Q) 
\]
for all lattice polytopes $P,Q$ such that $P\cup Q$ (and thus also $P\cap Q$) are lattice polytopes. A valuation $\varphi$ is called $\mathbf{translation-invariant}$ if $\varphi (P+t)=\varphi (P)$ for all lattice polytopes $P$ and all $t$ in the integer lattice. The volume and the number of lattice points in a lattice polytope present examples of such valuations. A multivariate version of Theorem~\ref{thm:Ehrhart} was proved by Bernstein~\cite{Bernstein} for the number of lattice points in Minkowski sums of lattice polytopes, and, more generally, by McMullen
~\cite{McMullen} for arbitrary translation-invariant valuations on lattice polytopes. The following result is often referred to as the Bernstein-McMullen theorem.

\begin{theorem}[{\cite[Theorem 6]{McMullen}}]\label{thm:Bernstein-McMullen}
Let $P_1,\ldots, P_k$ be lattice polytopes and let $\varphi$ be a translation-invariant valuation. Then $\varphi (n_1P_1+n_2P_2+\cdots +n_kP_k)$ agrees with a polynomial $\varphi _{P_1,\ldots, P_k}(n_1,\ldots, n_k)$ of total degree at most $\dim (P_1+\cdots +P_k)$ for all integers $n_1,\ldots, n_k\geq 0$.
\end{theorem}

The following extension of Theorem~\ref{thm:Bernstein-McMullen} complements results in~\cite{CombinatorialPositive,jochemko2015PhD} where a multivariate Ehrhart-Macdonald reciprocity was established and generalizes results by Ardila, Benedetti and Doker~\cite[Proposition 3.2]{Ardila} from volumes to translation-invariant valuations using a similar argument.

\begin{proposition}\label{lem:negativeEntries}
Let $P_1,\ldots, P_k, Q_1,\ldots, Q_\ell$ be lattice polytopes and let $n_1,\ldots, n_k, m_1,\ldots, m_\ell\geq 0$ be integers such that $Q=m_1Q_1+\cdots +m_\ell Q_\ell$ is a Minkowski summand of $P=n_1P_1+\cdots+n_kP_k$. Let $\varphi$ be a translation-invariant valuation. Then
\[
\varphi (P-Q)=\varphi _{P_1,\ldots, P_k,Q_1,\ldots, Q_\ell}(n_1,\ldots, n_k,-m_1,\ldots, -m_\ell)\, .
\]
\end{proposition}
\begin{proof}
By Theorem~\ref{thm:Bernstein-McMullen} the number of lattice points in $(P-Q)+tQ$ agrees with a polynomial for all integers $t\geq 0$. Let $f(t)$ denote this polynomial. On the other hand, since $(P-Q)+Q=P$, we obtain
\begin{eqnarray*}
\varphi ((P-Q)+tQ)&=&\varphi (n_1P_1+\cdots n_kP_k+(t-1)m_1Q_1+\cdots +(t-1)m_\ell Q_\ell)\\
&=&\varphi _{P_1,\ldots, P_k,Q_1,\ldots, Q_\ell}(n_1,\ldots, n_k, (t-1)m_1,\ldots, (t-1)m_\ell) 
\end{eqnarray*}
for all $t\geq 1$, again by Theorem~\ref{thm:Bernstein-McMullen}. Since two polynomials which agree infinitely many times must be equal we conclude 
\[
\varphi (P-Q)=f(0)=\varphi_{P_1,\ldots, P_k,Q_1,\ldots, Q_\ell}(n_1,\ldots, n_k,-m_1,\ldots, -m_\ell) \, ,
\]
as desired.
\end{proof}

The following expression for the number of lattice points in generalized permutahedra $\sum _{I\subseteq [d]} y_I \Delta _I$ has been proved by Postnikov~\cite[Theorem 11.3]{Postnikov} in the case when all coefficients $y_I$ are nonnegative integers. Proposition~\ref{lem:negativeEntries} allows us to extend formula~\eqref{eq:posPostnikov} to signed Minkowski sums and thus to all generalized permutahedra.

\begin{corollary}\label{cor:extpostnikov}
For all integer vectors $\{y_I\}_{I\subseteq [d]}$ that satisfy Equations~\eqref{eq:supermodtrans} in Theorem~\ref{thm:chargenperm} 
\begin{equation}\label{eq:posPostnikov}
    \left\lvert\sum _{I\subseteq [d]}y_I\Delta _I \cap \mathbb{Z}^d\right\rvert = \sum _{\mathbf{a}}{y_{[d]}+a_{[d]}\choose a_{[d]}}\prod _{I\subset [d]}{y_{I}+a_{I}-1\choose a_{I}}
\end{equation}
where the sum is over all non negative integer vectors $\{a_I\}_{I\subseteq [d]}$ such that $\sum _{I\subseteq [d]} a_I=d-1$ and for all $M\subseteq 2^{[d]}$ we have
\[
\left\lvert \bigcup _{J\in M} J\right\rvert \geq 1+\sum _{J\in M}a_J \, .
\]
\end{corollary}
\begin{proof}
By the Bernstein-McMullen Theorem~\ref{thm:Bernstein-McMullen}, the number of lattice points in $\sum _{I\subseteq [d]} y_I \Delta _I$ is given by a polynomial for all integers $y_I\geq 0$. Indeed, the right hand side of Equation~\eqref{eq:posPostnikov}, which was proved in~\cite[Theorem 11.3]{Postnikov} in this case, is a polynomial in the coefficients $y_I$, $I\subseteq [d]$, since
\[
{x\choose k}=1/k!\cdot x \cdot (x-1)\cdots (x-k+1) \, 
\]
for all nonnegative integers $x\geq 0$. By Theorem~\ref{thm:chargenperm}, an integer vector $\{y_I\}_{I\subseteq [d]}$ satisfies the Equations~\eqref{eq:supermodtrans} if and only if $\sum _{I\subseteq [d]} y_I \Delta _I$ defines a generalized permutahedra and this holds if and only if $Q:=\sum _{I\colon y_I<0}(-y_I)\Delta _I$ is a Minkowski summand of $P:=\sum _{I\colon y_I\geq 0}y_I\Delta _I$. Thus, by Proposition~\ref{lem:negativeEntries}, the polynomial expression for the number of lattice points in $\sum _{I\subseteq [d]} y_I \Delta _I$ given by Equation~\eqref{eq:posPostnikov} extends to all vectors $\{y_I\}_{I\subseteq [d]}$ satisfying the Equations~\eqref{eq:supermodtrans}.
\end{proof}

\subsection{Matroid polytopes} 
In this section we apply our results to matroid polytopes and matroid independent set polytopes to obtain inequalities for the beta invariant of a matroid. Let $M$ be a matroid on a groundset $E$ with rank function $r$. The \textbf{matroid polytope} $P_M$ is a polytope that is defined as the convex hull of all indicator functions of bases of $M$. The \textbf{beta invariant}~\cite{betainvariant} of $M$ is defined as 
\[
\beta (M) \ = \ (-1)^{r(M)}\sum _{X\subseteq E}(-1)^{|X|}r(X) \, .
\]
In~\cite{Ardila} a signed version, the \textbf{signed beta invariant}, 
\[
\tilde{\beta}(M) \ = \ (-1)^{r(M)+1}\beta (M) \, 
\]
was introduced in order to express the matroid polytope as a signed Minkowski sum of standard simplices 

\begin{proposition}[{\cite{Ardila}}]\label{prop:matroids}
Let $M$ be a matroid of rank $r$ on $E$ and let $P_M$ be its matroid polytope. Then
\[
P_M \ = \ \sum _{A\subseteq E} \tilde{\beta}(M/A)\Delta _{E-A} \, .
\]
\end{proposition}
As a  consequence of Theorem~\ref{thm:Ehrhartpositive} together with Proposition~\ref{prop:matroids} and recalling that $E_1(\Delta _i)=1+\frac{1}{2}+\cdots +\frac{1}{i-1}$ we obtain the following inequality for signed beta invariants of contractions.
\begin{corollary}\label{cor:matroids}
Let $M$ be a matroid with groundset $E$. Then
\[
\sum _{A\subseteq E} h_{|E-A|-1}\tilde{\beta}(M/A)  \geq \ 0 \, ,
\]
where $h_i:=1+\frac{1}{2}+\cdots + \frac{1}{i}$.
\end{corollary}
The \textbf{independent set polytope} $I_M$ of a matroid $M$ is defined as the convex hull of indicator functions of all independent sets of $M$. For $I\subseteq E$ let 
\[
D_I \ = \ \conv( \{0\}\cup \{e_i \colon i\in I\}) \, .
\]
In \cite{Ardila} these simplices were used to express the matroid independence polytope as a signed Minkowski sum.
\begin{proposition}[{\cite{Ardila}}]\label{prop:matroidsindep}
Let $M$ be a matroid of rank $r$ on $E$ and let $I_M$ be its independent set polytope. Then
\[
I_M \ = \ \sum _{A\subseteq E} \tilde{\beta}(M/A)D _{E-A} \, .
\]
\end{proposition}
\begin{corollary}\label{cor:matroidsindep}
Let $M$ be a matroid with groundset $E$. Then
\[
\sum _{A\subseteq E} h_{|E-A|}\tilde{\beta}(M/A)  \geq \ 0 \, ,
\]
where $h_i:=1+\frac{1}{2}+\cdots + \frac{1}{i}$.
\end{corollary}
\begin{proof}
After a lattice preserving affine transformation $\mathbb{R}^{|E|}\rightarrow \mathbb{R}^{|E|+1}$, $e_i\mapsto e_i$, $0\mapsto e_{|E|+1}$, $I_M$ is a generalized permutahedron and $D_I$ are standard simplices. The proof follows then from Theorem~\ref{thm:Ehrhartpositive}.
\end{proof}

\subsection{Solid angles}
We conclude by applying our results of the previous chapters to a close relative of the Ehrhart polynomial, the solid angle polynomial of a lattice polytope. Let $q\in \mathbb{R}^d$ be a point, $P\subseteq \mathbb{R}^d$ be a polytope and let $\mathcal{B}_\epsilon (q)$ denote the ball with radius $\epsilon$ centered at $q$. The \textbf{solid angle} of $q$ with respect to $P$ is defined by
\[
\omega _q (P) \ = \ \lim _{\epsilon \rightarrow 0} \frac{\vol (P\cap \mathcal{B}_\epsilon (q) )}{\vol \mathcal{B}_\epsilon (q)} \, .
\]
We note that the function $q\mapsto \omega _q (P)$ is constant on relative interiors of the faces of $P$. In particular, if $q\not \in P$ then $\omega _q (P) = 0$, if $q$ is in the interior of $P$ then $\omega _q (P)=1$ and if $q$ lies inside the relative interior of a facet then $\omega _q (P)=\frac{1}{2}$. The \textbf{solid angle sum} of $P$ is defined by 
\[
A(P) \ = \ \sum _{q\in \mathbb{Z}^d} \omega _q (P) \, .
\]
By an analog of Ehrhart's Theorem (Theorem~\ref{thm:Ehrhart}) for solid-angle sums due to Macdonald~\cite{Macdonald} $A(P)=A_0(P)+A_1(P)n+\cdots +A_d(P)n^d$ is a polynomial for all lattice polytopes $P$. This follows also from the Bernstein-McMullen Theorem~\ref{thm:Bernstein-McMullen} since $A(P)$ is a translation-invariant valuation (see, e.g., \cite{positivitysolid}). Indeed, since $\omega _p(P)$ is constant on relative interiors of faces
\begin{equation}\label{eq:solidangle}
A(nP) \ = \ \sum _{F\subseteq P}\sum _{q\in \relint F\cap \mathbb{Z}^d} \omega _q (nP) = \sum _{F\subseteq P}\omega _F (P)E_{\relint F}(n) \, ,
\end{equation}
where the first sum is over all faces $F$ of $P$, $\omega _F(P)$ is the solid angle of a point in the relative interior of $F$ and $E_{\relint F}(n)=|\relint nF \cap \mathbb{Z}^d|$ is the Ehrhart polynomial of the relative interior of $F$ (see \cite[Lemma 13.2]{BeckRobins}). For lattice polygons $P$ in $\mathbb{R}^2$, the solid-angle sum $A(P)$ agrees with the area, $\area (P)$, of the polygon~\cite[Corollary 13.11]{BeckRobins}. In particular, $A(nP)=\area (P)n^2$ has only nonnegative coefficients. As in the case of Ehrhart polynomials, for polytopes $P$ of higher dimension the coefficients $A_i(P)$ can be negative in general~\cite[Proposition 1]{positivitysolid}, even in dimension $3$. We supplement this result by showing that for the class of generalized permutahedra, unlike the case of Ehrhart polynomials, the linear terms of solid angle polynomials can be negative. Here, we view generalized permutahedra in $\mathcal{P}_d$ as polytopes in $\{\mathbf{x}\in \mathbb{R}^d\colon \sum x_i=\ell\}$ for some $\ell\in\mathbb{Z}$. 
\begin{proposition} Let $Q\in \mathcal{P}_4$ be the $3$-dimensional generalized permutahedron defined by
\[Q = \underset{\stackrel{I \subseteq [4]}{|I|=2}}{\sum} \Delta_{I} -  \Delta_{4} \, .\]
Then $A_1 (Q)<0$. In particular, there is a $3$-dimensional generalized permutahedron in $\mathbb{R}^4$ such that the linear term of its solid angle polynomial is negative. 
\end{proposition}
\begin{proof}
It is easy to check that the coefficients in the signed Minkowski sum by which $Q$ is given satisfy the inequalities~\eqref{eq:supermodtrans}, and therefore, by Theorem~\ref{thm:chargenperm}, $Q$ is a generalized permutahedron.
 Since the solid-angle sum is a translation-invariant valuation and by observing that Lemma~\ref{lem:linearityEhr} and its proof in \cite{MinkowskiValuations} via the Bernstein-McMullen Theorem~\ref{thm:Bernstein-McMullen} carries over verbatim to translation-invariant valuations we see that the linear term $A_1(P)$ is Minkowski additive.
By definition, $A(P)=0$ whenever $\dim P<3$ and therefore $A_1(\Delta_{I})=0$ for all $I\subseteq [4]$ with $|I|<4$. In particular, by Minkowski additivity, we have $A_1(Q)=-A_1(\Delta _4)$. It thus suffices to prove $A_1(\Delta _4)>0$. By \eqref{eq:solidangle},
\begin{eqnarray}
A(n\Delta _4) \ &=& \ \alpha E_{\relint \Delta _4}+ 4\beta E_{\relint \Delta _3}+6\gamma E_{\relint \Delta _2}+4\delta E_{\relint \Delta _1}\\
\ &=& \ \alpha {n-1\choose 3}+ 4\beta {n-1\choose 2 }+6\gamma (n-1)+4\delta \, ,
\end{eqnarray}
where $\alpha,\beta,\gamma,\delta$ denote the solid angle of $\Delta _4$ at a lattice point in the interior, on a facet, on an edge and at a vertex, respectively. Inserting the values $\alpha =1$, $\beta=\frac{1}{2}$ and $\gamma =\frac{\cos^{-1}(\frac{1}{3})}{2\pi}$ (see, e.g., \cite{Coxeter1973})  we obtain
\[
A_1(P) \ = \ \frac{3}{\pi}\cos^{-1}(\frac{1}{3})-\frac{7}{6}\cong 0.00881298... >0  \, 
\]
as desired. This completes the proof.
\end{proof}

\textbf{Acknowledgements:} The authors are grateful to the anonymous referees for their careful reading of the manuscript and many insightful remarks that improved this paper. Furthermore, they would like to thank the Simons Institute of the Theory of Computing, Berkeley, for hosting the semester program ``Geometry of Polynomials'' in spring 2019 during which this project evolved. The first author was funded by a Microsoft Research Fellowship during the semester program. She was also partially supported by the Wallenberg AI, Autonomous Systems and Software Program (WASP) funded by the Knut and Alice Wallenberg
Foundation and by grant 2018-03968 from the Swedish Research Council. The second author was supported by BAP 2018-16 from the Mimar Sinan Fine Arts University and a 2018-20 BAGEP award from Bilim Akademisi. 
\\

\bibliographystyle{abbrv}
\bibliography{final}

\begin{thebibliography}{10}

\bibitem{Hopfmonoid}
M.~Aguiar and F.~Ardila.
\newblock Hopf monoids and generalized permutahedra.
\newblock {\em arXiv preprint arXiv:1709.07504}, 2017.

\bibitem{negative0}
P.~Alexandersson.
\newblock Polytopes and large counterexamples.
\newblock {\em Exp. Math.}, 28(1):115--120, 2019.

\bibitem{Ardila}
F.~Ardila, C.~Benedetti, and J.~Doker.
\newblock Matroid polytopes and their volumes.
\newblock {\em Discrete Comput. Geom.}, 43(4):841--854, 2010.

\bibitem{ACEP20}
F.~Ardila, F.~Castillo, C.~Eur, and A.~Postnikov.
\newblock Coxeter submodular functions and deformations of {C}oxeter
  permutahedra.
\newblock {\em Adv. Math.}, 365:107039, 36, 2020.

\bibitem{BeckRobins}
M.~Beck and S.~Robins.
\newblock {\em Computing the continuous discretely}.
\newblock Undergraduate Texts in Mathematics. Springer, New York, second
  edition, 2015.
\newblock Integer-point enumeration in polyhedra, With illustrations by David
  Austin.

\bibitem{positivitysolid}
M.~Beck, S.~Robins, and S.~V. Sam.
\newblock Positivity theorems for solid-angle polynomials.
\newblock {\em Beitr\"{a}ge Algebra Geom.}, 51(2):493--507, 2010.

\bibitem{Bernstein}
D.~N. Bernstein.
\newblock The number of lattice points in integer polyhedra.
\newblock {\em Funkcional. Anal. i Prilo\v{z}en.}, 10(3):72--73, 1976.

\bibitem{MinkowskiValuations}
K.~J. B\"{o}r\"{o}czky and M.~Ludwig.
\newblock Minkowski valuations on lattice polytopes.
\newblock {\em J. Eur. Math. Soc. (JEMS)}, 21(1):163--197, 2019.

\bibitem{CastilloThesis}
F.~Castillo.
\newblock {\em Local Ehrhart positivity}.
\newblock PhD thesis, 2017.

\bibitem{nested}
F.~Castillo and F.~Liu.
\newblock Deformation cones of nested braid fans.
\newblock {\em arXiv preprint arXiv:1710.01899}, 2017.

\bibitem{CastilloFu}
F.~Castillo and F.~Liu.
\newblock Berline-{V}ergne valuation and generalized permutohedra.
\newblock {\em Discrete Comput. Geom.}, 60(4):885--908, 2018.

\bibitem{castillo2019todd}
F.~Castillo and F.~Liu.
\newblock On the {T}odd class of the permutohedral variety.
\newblock {\em arXiv preprint arXiv:1909.09127}, 2019.

\bibitem{negative1}
F.~Castillo, F.~Liu, B.~Nill, and A.~Paffenholz.
\newblock Smooth polytopes with negative {E}hrhart coefficients.
\newblock {\em J. Combin. Theory Ser. A}, 160:316--331, 2018.

\bibitem{Coxeter1973}
H.~S.~M. Coxeter.
\newblock {\em Regular polytopes}.
\newblock Dover Publications, Inc., New York, third edition, 1973.

\bibitem{betainvariant}
H.~H. Crapo.
\newblock A higher invariant for matroids.
\newblock {\em J. Combinatorial Theory}, 2:406--417, 1967.

\bibitem{Danilov}
V.~I. Danilov and G.~A. Koshevoy.
\newblock Cores of cooperative games, superdifferentials of functions, and the
  {M}inkowski difference of sets.
\newblock {\em J. Math. Anal. Appl.}, 247(1):1--14, 2000.

\bibitem{DeLoera}
J.~A. De~Loera, D.~C. Haws, and M.~K\"{o}ppe.
\newblock Ehrhart polynomials of matroid polytopes and polymatroids.
\newblock {\em Discrete Comput. Geom.}, 42(4):670--702, 2009.

\bibitem{Edmonds}
J.~Edmonds.
\newblock Submodular functions, matroids, and certain polyhedra.
\newblock In {\em Combinatorial {S}tructures and their {A}pplications ({P}roc.
  {C}algary {I}nternat. {C}onf., {C}algary, {A}lta., 1969)}, pages 69--87.
  Gordon and Breach, New York, 1970.

\bibitem{Ehrhart}
E.~Ehrhart.
\newblock Sur les poly\`edres rationnels homoth\'{e}tiques \`a {$n$}
  dimensions.
\newblock {\em C. R. Acad. Sci. Paris}, 254:616--618, 1962.

\bibitem{ferronihypersimplices}
L.~Ferroni.
\newblock Hypersimplices are {E}hrhart positive.
\newblock {\em J. Combin. Theory Ser. A}, 178:105365, 13, 2021.

\bibitem{ferroni2021matroids}
L.~Ferroni.
\newblock Matroids are not {E}hrhart positive.
\newblock {\em arXiv preprint arXiv:2105.04465}, 2021.

\bibitem{fink2018schubert}
A.~Fink, K.~M{\'e}sz{\'a}ros, and A.~S. Dizier.
\newblock Schubert polynomials as integer point transforms of generalized
  permutahedra.
\newblock {\em Adv. Math.}, 332:465--475, 2018.

\bibitem{Polymatroids}
S.~Fujishige.
\newblock {\em Submodular functions and optimization}, volume~58 of {\em Annals
  of Discrete Mathematics}.
\newblock Elsevier B. V., Amsterdam, second edition, 2005.

\bibitem{Goemans}
M.~X. Goemans.
\newblock Smallest compact formulation for the permutahedron.
\newblock {\em Math. Program.}, 153(1, Ser. B):5--11, 2015.

\bibitem{Grunbaum}
B.~Gr\"{u}nbaum.
\newblock {\em Convex polytopes}, volume 221 of {\em Graduate Texts in
  Mathematics}.
\newblock Springer-Verlag, New York, second edition, 2003.
\newblock Prepared and with a preface by Volker Kaibel, Victor Klee and
  G\"{u}nter M. Ziegler.

\bibitem{Hadwiger}
H.~Hadwiger.
\newblock {\em Vorlesungen \"{u}ber {I}nhalt, {O}berfl\"{a}che und
  {I}soperimetrie}.
\newblock Springer-Verlag, Berlin-G\"{o}ttingen-Heidelberg, 1957.

\bibitem{negativeCoefficients}
T.~Hibi, A.~Higashitani, A.~Tsuchiya, and K.~Yoshida.
\newblock Ehrhart polynomials with negative coefficients.
\newblock {\em Graphs Combin.}, 35(1):363--371, 2019.

\bibitem{Lange}
C.~Hohlweg, C.~E. M.~C. Lange, and H.~Thomas.
\newblock Permutahedra and generalized associahedra.
\newblock {\em Adv. Math.}, 226(1):608--640, 2011.

\bibitem{jochemko2015PhD}
K.~Jochemko.
\newblock {\em On the combinatorics of valuations}.
\newblock PhD thesis, 2015.

\bibitem{JochemkoSanyal}
K.~Jochemko and R.~Sanyal.
\newblock Combinatorial mixed valuations.
\newblock {\em Adv. Math.}, 319:630--652, 2017.

\bibitem{CombinatorialPositive}
K.~Jochemko and R.~Sanyal.
\newblock Combinatorial positivity of translation-invariant valuations and a
  discrete {H}adwiger theorem.
\newblock {\em J. Eur. Math. Soc. (JEMS)}, 20(9):2181--2208, 2018.

\bibitem{Kuipers}
J.~Kuipers, D.~Vermeulen, and M.~Voorneveld.
\newblock A generalization of the {S}hapley-{I}chiishi result.
\newblock {\em Internat. J. Game Theory}, 39(4):585--602, 2010.

\bibitem{FuCyclic}
F.~Liu.
\newblock Ehrhart polynomials of cyclic polytopes.
\newblock {\em J. Combin. Theory Ser. A}, 111(1):111--127, 2005.

\bibitem{liu2019positivity}
F.~Liu.
\newblock On positivity of {E}hrhart polynomials.
\newblock In {\em Recent Trends in Algebraic Combinatorics}, pages 189--237.
  Springer, 2019.

\bibitem{Macdonald}
I.~G. Macdonald.
\newblock Polynomials associated with finite cell-complexes.
\newblock {\em J. London Math. Soc. (2)}, 4:181--192, 1971.

\bibitem{McMullen}
P.~McMullen.
\newblock Valuations and {E}uler-type relations on certain classes of convex
  polytopes.
\newblock {\em Proc. London Math. Soc. (3)}, 35(1):113--135, 1977.

\bibitem{GraphicalModels}
F.~Mohammadi, C.~Uhler, C.~Wang, and J.~Yu.
\newblock Generalized permutohedra from probabilistic graphical models.
\newblock {\em SIAM J. Discrete Math.}, 32(1):64--93, 2018.

\bibitem{RankTests}
J.~Morton, L.~Pachter, A.~Shiu, B.~Sturmfels, and O.~Wienand.
\newblock Convex rank tests and semigraphoids.
\newblock {\em SIAM J. Discrete Math.}, 23(3):1117--1134, 2009.

\bibitem{M-convex}
K.~Murota.
\newblock {\em Discrete convex analysis}.
\newblock SIAM Monographs on Discrete Mathematics and Applications. Society for
  Industrial and Applied Mathematics (SIAM), Philadelphia, PA, 2003.

\bibitem{Postnikov}
A.~Postnikov.
\newblock Permutohedra, associahedra, and beyond.
\newblock {\em Int. Math. Res. Not. IMRN}, (6):1026--1106, 2009.

\bibitem{Faces}
A.~Postnikov, V.~Reiner, and L.~Williams.
\newblock Faces of generalized permutohedra.
\newblock {\em Doc. Math.}, 13:207--273, 2008.

\bibitem{Reeve}
J.~E. Reeve.
\newblock On the volume of lattice polyhedra.
\newblock {\em Proc. London Math. Soc. (3)}, 7:378--395, 1957.

\bibitem{CombOptB}
A.~Schrijver.
\newblock {\em Combinatorial optimization. {P}olyhedra and efficiency. {V}ol.
  {B}}, volume~24 of {\em Algorithms and Combinatorics}.
\newblock Springer-Verlag, Berlin, 2003.
\newblock Matroids, trees, stable sets, Chapters 39--69.

\bibitem{Shapley2}
L.~Shapley.
\newblock A value for n-person games.
\newblock pages 307--317, 1953.

\bibitem{Shapley}
L.~Shapley.
\newblock Cores of convex games.
\newblock {\em Internat. J. Game Theory}, 1:11--26; errata, ibid. 1 (1971/72),
  199, 1971/72.

\bibitem{zonotopes}
R.~P. Stanley.
\newblock A zonotope associated with graphical degree sequences.
\newblock In {\em Applied geometry and discrete mathematics}, volume~4 of {\em
  DIMACS Ser. Discrete Math. Theoret. Comput. Sci.}, pages 555--570. Amer.
  Math. Soc., Providence, RI, 1991.

\bibitem{Stanley}
R.~P. Stanley.
\newblock {\em Enumerative combinatorics. {V}olume 1}, volume~49 of {\em
  Cambridge Studies in Advanced Mathematics}.
\newblock Cambridge University Press, Cambridge, second edition, 2012.

\bibitem{studeny2016basic}
M.~Studen{\'y}.
\newblock Basic facts concerning extreme supermodular functions.
\newblock 2016.

\bibitem{StudenyKroupa}
M.~Studen\'{y} and T.~Kroupa.
\newblock Core-based criterion for extreme supermodular functions.
\newblock {\em Discrete Appl. Math.}, 206:122--151, 2016.

\bibitem{Ziegler}
G.~M. Ziegler.
\newblock {\em Lectures on polytopes}, volume 152 of {\em Graduate Texts in
  Mathematics}.
\newblock Springer-Verlag, New York, 1995.

\end{thebibliography}

\end{document}